\documentclass{amsart}
\def\cal#1{\mathcal{#1}}

          \usepackage{amssymb}
          \usepackage{amsmath}
          \usepackage{amsthm}
          \usepackage{amsfonts}
          \usepackage{enumerate}
          \usepackage{url}
          \usepackage{hyperref}
          \usepackage{color}

\usepackage[numbers,sort&compress]{natbib}
\newcommand{\comment}[1]{}

\def\indn#1{\{#1_n\}_{n\in\N}}
\newcommand{\proba}{\mathbb P}
\renewcommand{\P}{\mathbb P}
\newcommand{\esp}{{\mathbb E}}

\newcommand{\inv}{^{-1}}

\newcommand{\argmax}{{\rm{argmax}}}

\newcommand{\eqnh}{\begin{eqnarray*}}
\newcommand{\eqne}{\end{eqnarray*}}
\newcommand{\eqnhn}{\begin{eqnarray}}
\newcommand{\eqnen}{\end{eqnarray}}
\newcommand{\equh}{\begin{equation}}
\newcommand{\eque}{\end{equation}}

\def\summ#1#2#3{\sum_{#1 = #2}^{#3}}
\def\prodd#1#2#3{\prod_{#1 = #2}^{#3}}
\def\sif#1#2{\sum_{#1=#2}^\infty}
\def\bveee#1#2#3{\bigvee_{#1=#2}^{#3}}

\newcommand{\eqd}{\stackrel{d}{=}}

\def\ccbb#1{\left\{#1\right\}}

\def\pp#1{\left(#1\right)}

\def\mmid{\;\middle\vert\;}

\def\vv#1{{\boldsymbol #1}}

\def\qmand{\quad\mbox{ and }\quad}

\def\mfa{\mbox{ for all }}

\def\wt#1{\widetilde{#1}}

\def\R{{\mathbb R}}

\def\N{{\mathbb N}}

\def\calM{\mathcal M}
\def\calN{\mathcal N}

\def\cal#1{\mathcal{#1}}

\def\PD{{\rm PD}}

\usepackage{epsfig}
\usepackage{ifthen}
\usepackage[normalem]{ulem}

\newtheorem{Thm}{Theorem}[section]

\newtheorem{Prop}[Thm]{Proposition}
\newtheorem{Coro}[Thm]{Corollary}

\theoremstyle{definition}
\newtheorem{Rem}[Thm]{Remark}
\newtheorem{Def}[Thm]{Definition}
\newtheorem{Example}[Thm]{Example}

\numberwithin{equation}{section}

\title[Exchangeable hitting partitions]{Exchangeable random partitions from max-infinitely-divisible distributions}
\date{\today}
\author{Stilian Stoev}\address{Stilian Stoev\\Department of Statistics, University of Michigan, Ann Arbor, 
\\439 W. Hall, 1085 S. University\\
 Ann Arbor, MI 48109-1107, USA.}\email{sstoev@umich.edu}
\author{Yizao Wang}\address{Yizao Wang\\Department of Mathematical Sciences\\University of Cincinnati\\2815 Commons Way\\Cincinnati, OH, 45221-0025, USA.}\email{yizao.wang@uc.edu}

\begin{document}\sloppy
\begin{abstract}
The hitting partitions are random partitions that arise from the investigation of so-called hitting scenarios of max-infinitely-divisible (max-i.d.)~distributions. We study a  class of max-i.d.~laws
with exchangeable hitting partitions obtained by size-biased sampling from the jumps of a L\'evy subordinator. We obtain explicit formulae for the distributions of these partitions 
in the case of the multivariate $\alpha$-logistic and another family of exchangeable max-i.d.\ distributions. Specifically, the hitting partitions for these two cases are shown to coincide 
with the well-known Poisson--Dirichlet 
 partitions
 ${\rm PD}(\alpha,0),\ \alpha\in (0,1)$ and ${\rm PD}(0,\theta),\ \theta>0$.
 \end{abstract}

\maketitle

\section{Introduction}

Recently there has been a renewed interest in the study of multivariate records in extreme value theory (e.g.~\citep{goldie89records,gnedin07chain,hashorva05multiple} and references therein), 
motivated especially by the latest advances on the so-called 
{\em hitting scenarios} for extremal events. 
The notion of a hitting scenario originated from the 
investigation of the conditional laws of max-stable processes \citep{wang11conditional,dombry13regular}.
In words, a max-stable process can be represented as the pointwise maximum of a family of infinite conditionally independent stochastic processes, and hitting scenarios are introduced to describe whether pointwise maxima at various locations are contributed by a single underlying stochastic process. This notion plays a crucial role in simulation methods for max-stable processes.
 Hitting scenarios also arise naturally in the expression of the likelihood for max-stable models \cite{stephenson05exploiting,wadsworth13new}, and 
  in random tessellations determined by max-stable processes \citep{dombry18random}.
The latest advances on hitting scenarios are motivated by their connections to concurrence probabilities, and the framework can be naturally stated in the language of random partitions \citep{dombry18multivariate,dombry17probabilities}.

Our focus here is on the probabilistic aspects of the {\em hitting partitions} of multivariate max-stable distributions, recently introduced in \citep{dombry17probabilities}. 
A hitting partition can be viewed, more generally, as a random partition derived from the hitting scenario of a max-stable process 
\citep{dehaan84spectral,wang10structure}. We focus however on finite-dimensional distributions most of the time for the sake of simplicity.
We begin with recalling the definition. First, let $\{\xi_\ell\}_{\ell\in\N}$ be a measurable enumeration of points of a Poisson point process on $\R_+$ with intensity $\nu_\alpha(dx) :=\alpha x^{-\alpha-1}dx$ for some $\alpha>0$, and let $\{\vv Y_\ell\}_{\ell\in\N}$ be i.i.d.~copies of certain non-negative random vector $\vv Y =(Y_1,\dots,Y_n)$ with finite $\alpha$-moments.  
Throughout, $\N = \{1,2,\dots\}$ and $\N_0 = \N\cup\{0\}$.
Then, it is well known that the vector
\equh\label{eq:frechet}
\vv Z \equiv (Z_1,\dots,Z_n) \equiv \pp{\bveee \ell1\infty  \xi_\ell Y_{\ell,1},\dots,\bveee \ell1\infty \xi_\ell Y_{\ell,n}}
\eque
has a multivariate max-stable $\alpha$-Fr\'echet distribution \citep{dehaan84spectral}.  Throughout, we write $\vee\equiv \max$. 
Recall that a random variable $Z$ is said to be $\alpha$-Fr\'echet if
\[
\P(Z \le x) = \exp(- \sigma^\alpha /x^{\alpha}),\ \quad\mbox{ for }x\in (0,\infty),
\]
with  scale parameter $\sigma>0$. The vector $\vv Z$ is max-stable with $\alpha$-Fr\'echet
marginals if all its non-negative max-linear combinations $\bveee k1na_kZ_k$ are $\alpha$-Fr\'echet distributed, for 
all $a_k\ge 0,\ k=1,\dots,n$.
The representation \eqref{eq:frechet} is an instance of the so-called Le Page-type 
series representations in the special case of the semi-group $(\mathbb R^d,\vee)$ 
 \citep{davydov08strictly}. 

Given a max-stable Fr\'echet random vector with the representation \eqref{eq:frechet}, the induced hitting partition is determined as follows. Set 
\[
\ell^*(k) := \argmax_{\ell\in\N} \ccbb{\xi_\ell Y_{\ell,k}}, k=1,\dots,n.
\]
Note that $\{\xi_\ell Y_{\ell,k},\ \ell\in \N\}$ is a simple Poisson process on $(0,\infty)$ and hence with probability one, $\ell^*(k)$ is uniquely determined for every $k$; we restrict ourselves to this event 
from now on.  

The hitting partition of $[n] \equiv \{1,\dots,n\}$, for $n\in\N$,
 denoted by $\Pi_n$, is the random partition of equivalence classes 
induced by the relation
\equh\label{eq:partition}
i\sim j\quad \mbox{ if and only if }\quad \ell^*(i) = \ell^*(j),\mfa  i,j\in[n].
\eque
Here, $i\sim j$ reads as $i$ and $j$ are in the same block of the partition. Recall that a partition of $[n]$ is a collection of disjoint sets, 
the union of which is $[n]$. 

Thus far, most of the research on the hitting partitions
 has focused on the so-called {\em concurrence probability}, that is, the probability of the event that the hitting partition $\Pi_n$ 
consists of a single block \citep{dombry17probabilities,dombry18multivariate}.  The concurrence probability has the following expression
\[
p(n)\equiv \proba\pp{\Pi_n = \{[n]\}} 
= \esp\pp{\frac1{\esp(  \bveee k1n (Y_k^*/Y_k )^\alpha \mid \vv Y)}},
\]
where ${\vv Y}^*$ is an independent copy of $\vv Y$. This result was established in  
\citep[Theorem 2]{dombry17probabilities} by 
using the Slyvniak--Mecke formula, and the same method in principle could yield
formulae for the entire probability distribution of the hitting partition (see, e.g.~\citep{dombry13regular}). 
The general expressions are however 
neither explicit nor intuitive.

The motivation of this paper is to study specific choices of $\vv Y$, where the induced hitting partition has an explicit probability mass function. 
Our starting point is an example from a very recent paper
 \citep[Example 3]{dombry17probabilities}, where $\alpha\in (0,1)$ and $\vv Y = (Y_1,\dots,Y_n)$ has i.i.d.~$1$-Fr\'echet 
components.  Then, the distribution of  $\vv Z = (Z_1,\dots,Z_n)$ in \eqref{eq:frechet} becomes 
multivariate $\alpha$-logistic 
(see Example \ref{example:logistic} below).  In this case, the 
concurrence probability 
has a simple-looking formula
\begin{equation}\label{e:concurrence} 
p(n) = \prodd k1{n-1}\left(1-\frac \alpha k\right) \equiv \frac{\prod_{k=1}^{n-1} (k - \alpha) }{(n-1)!}, 
n\in\N.
\end{equation}

In this paper, first we
 explain this formula by showing that the hitting partition in this case is the {\em exchangeable} random partition induced by the Poisson--Dirichlet distribution with parameters $(\alpha,0)$. 
Poisson--Dirichlet distributions and exchangeable random partitions are 
fundamental objects in combinatorial stochastic processes, with numerous applications, notably in nonparametric inference and population genetics  \citep{pitman06combinatorial,berestycki09recent}. 
An outstanding family of {\em exchangeable} random partitions are the ones induced by the Poisson--Dirichlet 
distribution with parameters $\alpha,\theta$, 
referred to as the $\PD(\alpha,\theta)$ partitions for short. 
The legitimate values of the parameters are $\alpha<0, \theta = -m\alpha$ for some $m\in\N$ or $\alpha\in[0,1],\theta>-\alpha$. 
For any selected partition of $[n]$ with block sizes $n_1,\dots,n_k$ (such that $n_1+\cdots+n_k = n, n_1,\dots,n_k\ge 1$), the probability of a $\PD(\alpha,\theta)$ taking the value of this partition equals
 \begin{equation}\label{e:PD_law}
 p_{\alpha,\theta}(n_1,\dots,n_k)  = \frac{(\theta+\alpha)_{k-1\uparrow\alpha}\prodd i1k(1-\alpha)_{n_i-1\uparrow 1}}{(\theta+1)_{n-1\uparrow \alpha}}, 
 \end{equation}
 where $(x)_{m\uparrow\alpha} := \prod_{k=0}^{m-1} (x+k\alpha).$  (See \citep[Theorem 3.2, Definition 3.3]{pitman06combinatorial}.)
The Poisson--Dirichlet random partitions are actually  {exchangeable random partitions of $\N$}, although we focus 
on their restriction to $[n]$ \eqref{e:PD_law} most of the time.
\begin{Prop}\label{prop:1}
 The hitting partition $\Pi_n$ of the $\alpha$-logistic max-stable model, $\alpha\in(0,1)$, is a $\PD(\alpha,0)$ partition. 
\end{Prop} 

The result follows essentially from the {\em paintbox representation} of exchangeable random partitions, to be reviewed in Section \ref{sec:maxid}, 
and the max-stability property of Fr\'echet 
distributions. By recognizing the random weights in the paintbox representation as the (normalized) jumps of an $\alpha$-stable subordinator, we obtain that 
the hitting partition is in fact the $\PD(\alpha,0)$ partition. Thus, our method is completely different from the one applied in \citep{dombry17probabilities}.

Moreover, it turns out that the same idea of the proof can be applied to a larger family of hitting partitions, associated with a
class of max-infinitely-divisible (max-i.d.) distributions \citep{resnick87extreme}.  The latter are obtained as  in \eqref{eq:frechet}, but by considering a general intensity measure 
$\nu$ for the Poisson process $\{\xi_\ell,\ \ell \in \mathbb \N\}$ while keeping the independent $\alpha$-Fr\'echet marks $Y_{\ell,k}$'s.  The resulting 
max-i.d.\ distributions will be referred to as {\em sub-Fr\'echet max-i.d.~distributions}.

Our main result is Theorem \ref{thm:1}, which 
establishes a paintbox representation of the hitting partitions for all sub-Fr\'echet max-i.d.~distributions: these hitting partitions are precisely the exchangeable random partitions obtained via {\em size-biased sampling} of jumps from 
a subordinator with L\'evy measure $\nu$ \citep[Chapter 4.1]{pitman06combinatorial}. This representation allows us to identify another class of non-max-stable
sub-Fr\'echet max-i.d.~distributions, whose hitting partitions have the $\PD(0,\theta)$ laws, for $\theta>0$.

{\em The paper is organized as follows.} In Section \ref{sec:maxid} 
we review the background on exchangeable random partitions and prove the main result Theorem \ref{thm:1}. In Section \ref{sec:fdd} we provide related results on the sub-Fr\'echet max-i.d.~distributions.

\section{Hitting partitions of sub-Fr\'echet max-i.d.~distributions}\label{sec:maxid}
We shall consider a multivariate max-i.d.~distribution with the following representation
\equh\label{eq:maxid}
(\zeta_1,\dots,\zeta_n)\equiv \pp{\bveee \ell1\infty J_\ell Y_{\ell,1},\dots,\bveee \ell1\infty J_\ell Y_{\ell,n}},
\eque
where 
$\vv J\equiv \{J_\ell\}_{\ell\in\N}$ is a Poisson point process on $\R_+$ with intensity measure $\nu$, and $\{\vv Y_\ell\}_{\ell\in\N}$ are i.i.d.~random vectors independent from $\vv J$, each $\vv Y_\ell = (Y_{\ell,1},\dots,Y_{\ell,n})$ is a collection of independent $1$-Fr\'echet random variables, with scale parameters $\vv\sigma = (\sigma_1,\dots,\sigma_n)\in(0,\infty)^n$. 
The values of $\vv\sigma$ shall not have any impact in most of the discussions until 
Section \ref{sec:fdd}.

 We assume throughout that 
\[
 \nu(\R_+) = \infty \qmand \int_0^\infty(1\wedge x)\nu(dx)<\infty.
\]
 This ensures that the Poisson process $\vv J$ has infinitely many points, and
\[
 J_*:=\sif\ell1 J_\ell<\infty \mbox{ a.s.}
\]
In particular, the random variables the $\zeta_i$'s in \eqref{eq:maxid} are finite, almost surely.  This readily follows from the
max-stability property of the $Y_{\ell,k}$'s. Namely, 
by
 the fact that, $\mfa a_\ell\ge 0,\ \ell\in\N$ with $\sif \ell1 a_\ell<\infty$, we have
\equh\label{eq:max-stability}
\bigvee_{\ell=1}^\infty a_\ell Y_\ell\eqd \pp{\summ \ell1\infty a_\ell} Y_1 
\eque
for i.i.d.~$1$-Fr\'echet random variables $\{Y_\ell\}_{\ell\in\N}$.

We name the max-i.d.~random vector $\vv \zeta = (\zeta_1,\dots,\zeta_n)$ in  
\eqref{eq:maxid}, as a {\em sub-Fr\'echet} max-i.d.~random vector {\em with L\'evy measure $\nu$.} The terminology is inspired by the corresponding sub-stable distributions \citep{samorodnitsky94stable} (see \eqref{eq:mixture} below).

\begin{Def} For all
$n\in\N$,
 set
 \equh\label{eq:ell*}
 \ell^*(k) := \argmax_{\ell\in\N}{\ccbb{J_\ell Y_{\ell,k}}}, k=1,\dots,n,
 \eque
 and define now the random partition $\Pi_n$ of $[n]$ for 
 $n\in\N$
  by \eqref{eq:partition} as before. We refer to the so-defined random partition as the {\em hitting partition} of the max-i.d.~distribution in \eqref{eq:maxid}.
\end{Def}

As before,  $\ell^*(k)$ in \eqref{eq:ell*} is uniquely defined with probability one. Indeed, for every pair $\ell, \ell'\in\N, \ell\ne\ell'$, it is easy to see that with probability one $J_\ell Y_{\ell,k} \ne J_{\ell'} Y_{\ell',k}$ by conditioning on the values of $J_\ell, J_{\ell'}$, as the law of $Y_{\ell,k}$ has no atom.

\begin{Rem}
Following \cite[Chapter 5]{resnick87extreme}, given a non-negative max-i.d.~vector, say $\vv\zeta = (\zeta_1,\dots,\zeta_n)$, 
there exists a unique measure 
$\mu$ on $(\R_+^n,{\cal B}(\R_+^n))$, known as the L\'evy (or exponent) measure of $\vv\zeta$, such that for all
$\vv x = (x_1,\dots,x_n)\in\R_+^n$,
\begin{equation}\label{e:max-id-spec-measure}
\P( \zeta_k\le x_k,\ k=1,\dots,n) = \exp\pp{ -\mu([\vv 0,\vv x]^c)}.
\end{equation}
Thus, taking a Poisson point process $\vv\Psi = \{\vv\Psi_\ell\}_{\ell\in\N}$ on $\R_+^n$ with mean measure $\mu$, 
we obtain the stochastic representation
\[
(\zeta_1,\dots,\zeta_n)\eqd \pp{\sup_{\ell\in\N}\Psi_{\ell,1},\dots,\sup_{\ell\in\N}\Psi_{\ell,n}}.
\]
The law of the hitting partition depends only on
the law of $\vv \Psi$, which in turn is uniquely 
determined by the law of $\vv \zeta$
as argued in \citep{dombry17probabilities}. 
However, our starting point \eqref{eq:maxid} is the assumption that the Poisson point process $\vv\Psi$ 
has the following specific representation
\[
\{\vv\Psi_\ell\}_{\ell\in\N} \eqd \{J_\ell \vv Y_{\ell}\}_{\ell\in\N}.
\]
Many different spectral representations of this form are possible. Here, we assume that one exists where the random vector $\vv Y$ has
independent 1-Fr\'echet components, which plays a crucial role in deriving formulae of random partitions. 
See Remark \ref{rem:Bernstein} for an alternative characterization of the exponent measure $\mu$  in this case.
\end{Rem}

 Our main result is to show that in this framework, the random partition $\Pi_n$ is exchangeable, and
  it has the same law as a paintbox partition (a.k.a.~partition generated by random sampling) with random weights 
 \equh\label{eq:Pell}
 P_\ell := \frac{J_\ell}{\sum_{\ell'=1}^\infty J_{\ell'}}, \ell\in\N \qmand P_0:=0.
 \eque
 We first review the background of the paintbox construction 
 \citep{berestycki09recent,pitman06combinatorial}. 
 Recall that a paintbox partition with respect to weight $\vv s = (s_0,s_1,\dots)$ 
 such that $s_0\ge 0, s_1\ge s_2\ge\cdots \ge s_n\ge \cdots\ge 0$ and $\sif k0 s_k = 1$, is a canonical way to obtain exchangeable random 
 partitions of $\N$ as follows. Let $\indn X$ be i.i.d.~sampling from $\N_0$ with distribution $\proba(X_1 = \ell) = s_\ell, \ell\in\N_0$. Color the
 set of natural numbers $\N =\{1,2,\dots\}$ as follows. If $X_i>0$, we paint $i$ in color $X_i$, otherwise, all $i$'s with $X_i=0$ are colored in different
 colors that are also different from all other colors used in the paintbox.  Thus $\N$ is partitioned into disjoint blocks by 
 different colors.  Formally, this partition
 is induced by the equivalence relation $i\sim j$ if $X_i = X_j>0$ for all $i,j\in\N$.  Notice that 
  every
 $i\in\N$ such that $X_i = 0$ forms a singleton block by 
 itself.

 It is well known and easy to see that the so-obtained partition of $\N$ is exchangeable. Moreover, Kingman's representation theorem 
 \citep[Theorem 2.2]{pitman06combinatorial} says that every exchangeable random partition of $\N$ can be obtained by a paintbox partition 
 with possibly random weights $\vv s$. In this case, conditionally on $\vv s$,  $\indn X$ are i.i.d.~with distribution 
 $\proba(X_1 = \ell\mid \vv s) = s_\ell, \ell\in\N_0$.
 
 Therefore, a convenient way to characterize the law of an infinite exchangeable partition is to identify 
 the random weights of the corresponding paintbox partition. 
Our discussions focus on finite partitions: if a finite partition can be obtained by a paintbox partition with a finite number of i.i.d.~samplings, it is clearly exchangeable, with the law determined by the weights, and we still refer to it as a paintbox partition. 
 \begin{Thm}\label{thm:1} For all
 $n\in\N$,
  the hitting partition $\Pi_n$ associated to \eqref{eq:maxid} is an exchangeable random partition
  of $[n]$, which has the same law as a paintbox partition with random weights $\{P_\ell\}_{\ell\in\N_0}$ given by \eqref{eq:Pell}.
 \end{Thm}

\begin{proof}
To see this, we first observe that conditioning on $\vv J$, for each $k=1,\dots,n$, the distribution of $\ell^*(k)$ is determined by   $\P(\ell^*(k) = \ell\mid\vv J), \ell\in\N$, 
and conditionally on $\vv J$, we have that $\{\ell^*(k)\}_{k=1,\dots,n}$ are independent, since so are the $Y_{\ell,k}$'s. 
The probability of interest turns out to be independent from $k$. Indeed, we have
\begin{align}
\nonumber
\P(\ell^*(k) = \ell\mid\vv J) &= \proba\pp{J_\ell Y_{\ell,k}>\max_{\ell'\ne\ell} J_{\ell'}Y_{\ell',k}\mmid\vv J} 
 \\ &
 = 
\proba\pp{J_\ell{Y_{1,k}}>\pp{\sum_{\ell'\ne\ell}J_{\ell'}}Y_{2,k}\mmid\vv J} 
 \label{e:Thm1-1}
 \\
\label{e:Thm1-2}
& = \frac{J_\ell}{\sum_{\ell'=1}^\infty J_{\ell'}} = P_\ell, \ell\in\N.
\end{align}
Relation \eqref{e:Thm1-1} follows from the max-stability property  \eqref{eq:max-stability} of the Fr\'echet distribution, while
Relation \eqref{e:Thm1-2} follows from the property 
$$
 \proba(aY_1>bY_2) = \frac{a}{a+b},
 $$
 valid for all $a,b>0$ where $Y_1$ and $Y_2$ are i.i.d.~1-Fr\'echet random variables.  
 This completes the proof.
 \end{proof}
 The aforementioned framework of exchangeable random partitions based on jump sizes of subordinators,  via 
 \eqref{eq:Pell},
  is well known.  In fact, 
{if} the L\'evy measure $\nu$ has a density $\rho$, then under mild conditions 
explicit formula for the random partition generalizing \eqref{e:PD_law} is available in terms of $\rho$. See \citep{pitman03poisson} and \citep[Exercise 4.1.2]{pitman06combinatorial}. To keep the presentation short, we shall instead explain only two special examples in full detail here.

Recall that the
 Poisson--Dirichlet distribution refers to a two-parameter family of {\em ranked frequencies} of $\{P_\ell\}_{\ell\in\N}$, indexed by $(\alpha,\theta)$ with either $\alpha<0, \theta = -m\alpha$ for some $m\in\N$, or $\alpha\in[0,1],\theta>-\alpha$. When the frequencies are ordered in {\em size-biased order}, the corresponding law of the same two-parameter family is known as the Griffiths--Engen--McCloskey (GEM) distribution. The size-biased frequencies, denoted by $\{\wt P_\ell\}_{\ell\in\N}$, have the representation
\[
(\wt P_1,\wt P_2,\wt P_3,\dots) \eqd (W_1,(1-W_1)W_2,(1-W_1)(1-W_2)W_3,\dots),
\]
where $\{W_\ell\}_{\ell\in\N}$ are independent beta random variables, each with parameters $(1-\alpha,\theta+\ell\alpha)$. 
Here, formally the size-biased frequencies are defined iteratively as follows: given the probability on $\N$ determined by $\{P_\ell\}_{\ell\in\N}$, consider a sequence of i.i.d.~sampling from this  probability, and let $\wt P_1$ denote the probability of the first label sampled,  $\wt P_2$ denote the second new label sampled, and so on.

In particular, explicit examples relating random weights from subordinators via \eqref{eq:Pell} to Poisson--Dirichlet distributions have been well known, and we shall make use 
of
 the following two from them (see \citep[Chapter 4.2]{pitman06combinatorial}): 
\begin{enumerate}[(i)]
\item If $\nu(dx) = \alpha x^{-\alpha-1}dx$, $\alpha\in(0,1)$, then the subordinator is an $\alpha$-stable subordinator. The ranked frequencies have the law of $\PD(\alpha,0)$. 
\item If $\nu(dx) = \theta x^{-1}e^{-x}dx$, $\theta>0$, then the subordinator is a Gamma process. The ranked frequencies  have the law of $\PD(0,\theta)$. 
\end{enumerate}
The next corollary follows immediately, including Proposition \ref{prop:1} as the first case.

\begin{Coro}\label{c:PD}
For the sub-Fr\'echet max-i.d.~distribution~\eqref{eq:maxid} with L\'evy measure $\nu$, the induced hitting partition $\Pi_n$ is:
\begin{enumerate}[(i)]
\item  $\PD(\alpha,0),\ \alpha \in (0,1)$  if $\nu(dw) =\alpha w^{-\alpha-1} dw$.
\item $\PD(0,\theta),\ \theta>0$  if $\nu(dw) = \theta w\inv e^{-w}dw$.
\end{enumerate}
\end{Coro}

\begin{Rem}
The paintbox argument in the proof of Theorem \ref{thm:1} applies without change to the case where \eqref{eq:maxid} is an infinite
max-i.d.~sequence indexed by $\N$.  In this case, one obtains exchangeable partitions of $\N$.  We stated Theorem \ref{thm:1} in the finite-dimensional 
setting of partitions on $[n]$ for simplicity and in order to draw connections to the exiting results on the concurrence probability (e.g.~formula \eqref{e:concurrence}).
\end{Rem}

\begin{Rem}
The class of hitting partitions arising from \eqref{eq:maxid} does not contain all exchangeable random partitions. By allowing dependence among, and/or other types 
of distributions of, $Y_{\ell,1},\dots,Y_{\ell,n}$, one could obtain other exchangeable random partitions by the same mechanism. In particular, 
it seems that  $\PD(\alpha,\theta)$ partitions for other choices of the parameters do not appear as the hitting partitions of sub-Fr\'echet max-i.d.~laws. 
\end{Rem}

\section{Distributions of sub-Fr\'echet max-i.d.~distributions}\label{sec:fdd}

So far, we have introduced a specific family of max-i.d.~distributions and shown 
that the induced hitting partitions have more explicit structure. In this section, we collect some facts 
on this family of max-i.d.~distributions. All the computations are straightforward and standard, and are hence omitted.

\begin{Prop}\label{prop:M} 
For $(\zeta_1,\dots,\zeta_n)$ as in \eqref{eq:maxid}, 
  \begin{equation}\label{eq:zeta_fdd}
  \P( \zeta_k\le x_k,\ k=1,\dots,n) = L_{J_*} \left( \sum_{k=1}^n \frac{\sigma_k}{x_k} \right), x_k\ge 0, k=1,\dots,n,
  \end{equation}
  where 
\begin{equation}\label{e:J-star-Laplace}
  L_{J_*}(t) := \esp \exp\pp{-t J_*} = \exp\pp{-\int_0^\infty (1-e^{-tx})\nu(dx)}, t>0
  \end{equation}
   is the Laplace transform of $J_*$. In particular, 
\equh\label{eq:mixture}
(\zeta_1,\dots,\zeta_n)\eqd  J_*(Y_{1,1},\dots,Y_{1,n}).
\eque
  \end{Prop}

 \begin{Rem}\label{rem:Bernstein} 
Relations \eqref{eq:zeta_fdd} and
\eqref{e:max-id-spec-measure} imply that a positive max-i.d.\ random vector $\vv \zeta$ is sub-Fr\'echet with exponent measure $\nu$ as in \eqref{eq:maxid}, if and only if
$$
\mu([\vv 0, \vv x]^c) =
g_\nu \left( \sum_{k=1}^n \frac{\sigma_k}{x_k} \right), \vv x\in(0,\infty)^n,
$$
where $g_\nu(t) \equiv -\log L_{J^*}(t) = \int_0^\infty (1-e^{-tx})\nu(dx),\ t\ge 0$ is the Laplace exponent of $J^*$
as in \eqref{e:J-star-Laplace} and $\vv \sigma\in(0,\infty)^n$. Note that $g_\nu$ is a Bernstein function in general.
\end{Rem}

\begin{Rem}
A priori, it is not obvious that for what functions $L_{J_*}$ the right-hand side of \eqref{eq:zeta_fdd} defines a valid multivariate distribution function.
An alternative proof of this fact  can be obtained from the perspective of Archimedean copula \cite{mcneil09multivariate,genest18class}.
The recent work of \cite{mai18extreme} uses related ideas on the elegant properties of strongly sum infinitely divisible laws to study a class of multivariate 
max-stable laws (formulated as extreme-value copula).  
\end{Rem}
  
  \begin{Example}\label{example:logistic}
  For the two cases of Corollary \ref{c:PD}, we have accordingly the following explicit formulae.
  \begin{enumerate}[(i)]
  \item 
  $L_{J_*}(t) = e^{-t^\alpha}$, and  \eqref{eq:zeta_fdd} reads as
\[
   \P( \zeta_k\le x_k,\ k=1,\dots,n) =\exp\pp{ - \left( \sum_{k=1}^n \frac{\sigma_k}{x_k} \right)^{\alpha} }.
\]
This is essentially the $\alpha$-logistic max-stable distribution in the literature, 
which is conventionally standardized to have $1$-Fr\'echet marginals with scale parameter $1$ 
(corresponding to $(\zeta_1^\alpha,\dots,\zeta_n^\alpha)$ here). 
This is the only sub-Fr\'echet max-i.d.\ model which is also max-stable.
\item  $L_{J_*}(t) = 1/(1+t)^\theta$,
and  \eqref{eq:zeta_fdd} reads as
\[
   \P( \zeta_k\le x_k,\ k=1,\dots,n) = \pp{1+ \sum_{k=1}^n \frac{\sigma_k}{x_k} }^{-\theta},\ \ x_k\ge 0, k=1,\dots,n.
\]
  \end{enumerate}
  \end{Example}

\begin{Rem}
The sub-Fr\'echet max-i.d.~distributions can be extended immediately to max-i.d.~random sup-measures, a topic which has raised some recent interest in the 
literature \citep{obrien90stationary,molchanov16max}. Indeed, the law of the corresponding random sup-measures $\calM$ 
 is uniquely determined by its finite-dimensional distributions \citep[Theorem 11.5]{vervaat97random}, which essentially we have already computed 
 in Proposition \ref{prop:M}. 

Namely, we can define a family of max-i.d.~random sup-measures on a generic measurable space $(E,\cal E)$ equipped with a $\sigma$-finite measure $\mu$, in the form of
\[
  {\cal M}(\cdot)\equiv {\cal M}_{\mu,\nu}(\cdot) := \bigvee_{\ell= 1}^\infty J_\ell{\cal N}_\ell(\cdot),
\]
  where $\vv J\equiv \{J_\ell\}_{\ell\in\N}$ a Poisson point process with mean measure $\nu$ as before and $\{\cal N_\ell\}_{\ell\in\N}$ are i.i.d.~independently scattered $1$-Fr\'echet random sup-measures on $(E,\cal E)$ with control measure $\mu$ \citep{stoev06extremal}, independent from $\vv J$. One can show that
  \[
  \calM \eqd J^*\calN_1,
  \]
  where $J^*$ is as before. 
  Both cases in the previous example can be extended to the corresponding max-i.d.~random sup-measures. Details are omitted.

\end{Rem}

\subsection*{Acknowledgements}
We thank two anonymous referees for comments that helped improve the paper.
SS's research was 
partially supported by the NSF FRG grant DMS-1462368.
YW's research was partially supported by NSA grant H98230-16-1-0322 and Army Research Laboratory grant W911NF-17-1-0006. 

 \bibliographystyle{apalike}
\bibliography{references}

\end{document}